\documentclass[11pt]{amsart}
\usepackage[centertags]{amsmath}
\usepackage{amsfonts,amssymb}
\usepackage{graphicx}
\usepackage{enumerate}
 \usepackage{url}
\usepackage{caption}
  \newtheorem{theorem}{Theorem}
  \newtheorem{corollary}{Corollary}

  \newtheorem{lemma}{Lemma}

\begin{document}

\title{Explicit bounds for Buchstab's function}

\author{Andreas Weingartner} 
\address{Department of Mathematics, Southern Utah University, 351 West University Boulevard, Cedar City, Utah 84720, USA} 
\email{weingartner@suu.edu} 

\begin{abstract}
Buchstab's function $\omega(u)$ describes the distribution of integers without small prime factors.
We establish numerically explicit upper and lower bounds for $\omega(u)$ that are easy to evaluate,
without the need to solve the delay differential equation numerically.
\end{abstract}

\maketitle

\section{Introduction}
Buchstab's function $\omega(u)$ is defined as the continuous solution, for $u>1$, to the delay differential equation
$(u \omega(u))'=\omega(u-1)$ with initial condition $\omega(u)=1/u$ for $1\le u\le 2$. 
If $\Phi(x,y)$ denotes the number of positive integers up to $x$ whose prime factors are all greater than $y$,
Buchstab \cite{Buch} proved that for fixed $u>1$ we have $\Phi(x,x^{1/u}) \sim x \omega(u)/\log (x^{1/u})$ as $x\to \infty$. 

As $u$ grows,  $\omega(u)$ oscillates about $e^{-\gamma}$ and converges to this value very rapidly, 
where $\gamma$ is Euler's constant.
Cheer and Goldston \cite{CG} calculate numerical values of the local extrema of $\omega(u)-e^{-\gamma}$ for $u<11$.
Those extreme values can be used to establish an unusual number of primes in short intervals.
Maier \cite{Maier} showed that (see \cite[(6)]{CG}), for $v>1$,
$$
\limsup_{x\to \infty} \frac{\pi(x+(\log x)^v )- \pi(x)}{(\log x)^{v-1}} \ge 1+e^\gamma \max_{u\ge v} (\omega(u)-e^{-\gamma}),
$$
$$
\liminf_{x\to \infty} \frac{\pi(x+(\log x)^v )- \pi(x)}{(\log x)^{v-1}} \le 1+e^\gamma \min_{u\ge v} (\omega(u)-e^{-\gamma}),
$$
where $\pi(x)$ denotes the number of primes up to $x$.
The purpose of this note is to establish numerically explicit upper and lower bounds for
$$
W(u):= \omega(u)-e^{-\gamma}.
$$

For $u\ge 1$, let $\zeta=\zeta(u) = \mu-\eta i$ be the unique complex solution in the range $\mu\ge 1$, $0< \eta < 3\pi/2$, of 
$$e^\zeta=-u\zeta .$$ 
Then $\zeta=\log(u\log u)-\pi i +o(1)$ as $u\to \infty$, by Lemma \ref{lemzeta}.

For $s\in \mathbb{C}\setminus (-\infty,0]$, define
$$
J(s):=\int_0^\infty \frac{e^{-s-t}}{s+t} dt =  \int_s^\infty \frac{e^{-z}}{z}dz,
$$
and extend the definition of $J(s)$ to the negative real axis by continuity from the upper half-plane
($J(s)$ jumps by $2\pi i$ along the negative real axis).
Let
$$
\Phi(u):=\frac{\exp\{-u\zeta+J(-\zeta)\}}{\sqrt{2 \pi J''(-\zeta)}} = \frac{\exp\{-u\zeta+J(-\zeta)\}}{\sqrt{2 \pi u(1-1/\zeta)}},
$$
where the square root denotes the principal branch. 
With a slightly different definition of $\zeta$ (namely $e^{\zeta^*}=1-u \zeta^*$), 
Hildebrand \cite[(2.5)']{Hil} shows
that 
\begin{equation}\label{HilEq}
W(u)=2 |\Phi(u)|\{\cos(\arg(\Phi(u)))+O(1/u)\}.
\end{equation}
Moreover, he finds \cite[Cor.~2]{Hil} that 
$$|\Phi(u)| = \rho(u) \exp\{-(1+o(1)) u\pi^2/(2\log^2 u)\} \qquad (u\to \infty),$$ 
where $\rho(u)$ is Dickman's function, and 
$$(\arg(\Phi(u))'\sim \pi \qquad  (u \to \infty).$$

We can improve the relative error $O(1/u)$ in \eqref{HilEq} if we replace $\Phi(u)$ by
$$
\Phi_1(u):= \Phi(u)\left(1-\frac{1}{12u}\right).
$$
Theorem \ref{thm1} provides numerically explicit estimates for $W(u)$, and hence for $\omega(u)$,
without the need to solve the delay differential equation numerically.
Since $\omega(u)=1/u$ for $1\le u \le 2$ and $\omega(u)=(1+\log(u-1))/u$ for $2 \le u \le 3$, we 
may assume $u\ge 3$ without any loss of generality. 

\begin{theorem}\label{thm1}
We have 
$$
W(u)=2 |\Phi_1(u)|\left\{\cos(\arg(\Phi_1(u)))+\theta_1(u)\right\},
$$
where $|\theta_1(u)|< \frac{1}{2u(u-1)}$ for $3\le u\le 6$ and $|\theta_1(u)|< \frac{1}{12u\log u}$ for $u\ge 6$.
\end{theorem}

\begin{corollary}\label{cor1}
We have $|W(u)|<2|\Phi(u)|$ for $u\ge 3$. 
\end{corollary}

To compute $\Phi(u)$, note that
$
J(s)=\Gamma(0,s),
$
the incomplete Gamma function (\texttt{Gamma[0,s]} in Mathematica, \texttt{incgam(0,s)} in PARI/GP), 
and
$
\zeta = \zeta(u) =-W_{1}\left(1/u\right),
$
where $W_{1}(s)$ is the $1$-branch
of the Lambert W function (\texttt{ProductLog[1,s]} in Mathematica, \texttt{lambertw(s,1)} in PARI/GP).
Table \ref{table1} shows numerical examples based on Theorems \ref{thm1} and \ref{thm2}.
\begin{table}[h] 
  \centering 
  \begin{tabular}{|r|l|l|r|} 
    \hline
    $u$ & $a$: Thm.~\ref{thm1} & $a$: Thm.~\ref{thm2} & $b$ \\ 
    \hline \hline
    6 & $+2.0$ & $+2.0$ & 7 \\ 
    \hline
10 & $-4.4$ & $-4.42$ & 14 \\ 
   \hline
    40 & $-1.40...$ & $-1.4080$ & 79 \\
    \hline
   100 & $+1.57...$ & $+1.57664$ & 241 \\
\hline
 1000 & $-1.7534...$ & $-1.7534550...$ & 3512 \\
\hline
 10000 & $+1.87930...$ & $+1.879306684...$ & 46123 \\
\hline
  \end{tabular}
   \caption{Values of $W(u) = a \cdot 10^{-b}$ from Theorems \ref{thm1} \& \ref{thm2}. 
Truncated values are shown with dots, others are rounded.} 
   \label{table1} 
\end{table}

Theorem \ref{thm2}, which is more precise than Theorem \ref{thm1}, involves the expression
$$
\Phi_2(u):=\Phi(u)(1+\alpha(u)),
$$
where 
\begin{equation*}
\begin{split}
\alpha(u)  := & \frac{1}{8} \frac{J^{(4)}(-\zeta)}{(J^{(2)}(-\zeta))^2} - \frac{5}{24} \frac{(J^{(3)}(-\zeta))^2}{(J^{(2)}(-\zeta))^3} \\
  = & -\frac{1}{12u} \cdot \frac{ \zeta^4-4 \zeta^3+\frac{13}{2} \zeta^2-2 \zeta+1}{ \zeta (\zeta-1)^3 }
 \sim -\frac{1}{12u}.
\end{split}
\end{equation*}
The derivatives $J^{(k)}(-\zeta)$ can easily be found from $J'(s)=-\frac{e^{-s}}{s}$ (see \eqref{eqJk}).

\begin{theorem}\label{thm2}
We have 
$$
W(u)=2 |\Phi_2(u)|\left\{\cos(\arg(\Phi_2(u)))+\theta_2(u)\right\},
$$
where $|\theta_2(u)|< 420 u^{-6}$ for $6 \le u\le 19$ and $|\theta_2(u)|< 0.005 u^{-2}$ for $u\ge 16$.
\end{theorem}

Figure \ref{fig1} shows the graphs of $\theta_1(u)$ from Theorem \ref{thm1}, 
 $\theta_2(u)$ from Theorem \ref{thm2}, 
and the bounds $\pm \frac{ 1}{12u\log u}$ from Theorem \ref{thm1}, for $6< u \le 12$. 
Theorem \ref{thm3} below exhibits an even smaller error term than Theorem \ref{thm2}, at the expense of a more 
complicated main term.

\begin{figure}[htbp]
    \centering
    \includegraphics[width=0.92\textwidth]{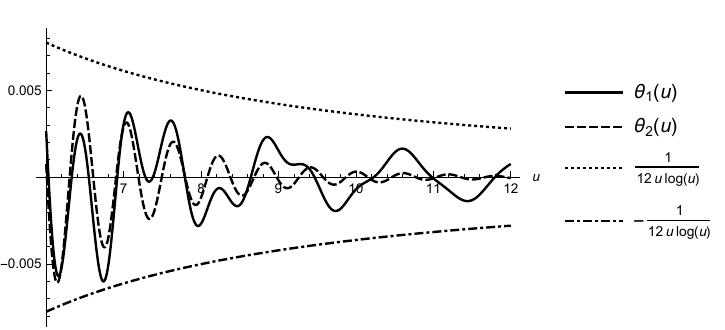}
    \caption{$\theta_1(u)$ and $\theta_2(u)$ for $6\le u \le 12$.}
    \label{fig1}
\end{figure}

Our theorems and their proofs mirror those in \cite{Wein}, where we found explicit estimates for 
Dickman's function $\rho(u)$. The definitions of $\Phi(u)$ and $\alpha(u)$ are analogous to 
the definitions of $\tilde{\rho}(u)$ and $\alpha(u)$ in  \cite{Wein},
with $J^{(k)}(-\zeta)$ replacing $I^{(k)}(\xi)$.
A key difference is that $\zeta$, $J^{(k)}(-\zeta)$ and $\Phi(u)$ are complex, 
while their counterparts $\xi$,  $I^{(k)}(\xi)$ and $\tilde{\rho}(u)$ in \cite{Wein} are real valued. 

Hildebrand \cite{Hil} derives \eqref{HilEq} by taking the limit of an asymptotic estimate for the solution to
$uf'(u)= \kappa(f(u)-f(u-1))$, as $\kappa \to -1^+$.
Tenenbaum \cite{Ten99} shows how \eqref{HilEq} can be deduced from \cite{HT}, where 
the asymptotic behavior of solutions to the general equation $u f'(u)+af(u)+bf(u-1)=0$
is described.

\section{Proof of Theorem \ref{thm2}}

Like Hildebrand \cite{Hil}, 
we use the saddle point method, although with a slightly different definition of $\zeta$.
Our choice of $\zeta$ ensures that, when evaluating the inverse Laplace integral,
the integrand $e^{J(s)+us}$ has a vanishing derivative exactly at $s=-\zeta$. 
Unlike in \cite{Hil}, we keep track of all constant factors 
and use a Taylor polynomial of order seven instead of three near the saddle point. 
For small values of $u$, we verify the result by calculating $W(u)$ numerically as described in Section \ref{SecNum}.

\subsection{Auxiliary results}

\begin{lemma}\label{lemlap}
The Laplace transform of $\omega(u)$, defined by $\widehat{\omega}(s):=\int_0^\infty \omega(u)e^{-us}du$ for $\mathrm{Re}(s) >0$, 
extends to a meromorphic function on $\mathbb{C}$ by $\widehat{\omega}(s)=e^{J(s)}-1$, with its only pole being a simple pole at the origin
with residue $e^{-\gamma}$. 
\end{lemma}
\begin{proof}
See Tenenbaum \cite[Thm.~III.6.7]{Ten}.
\end{proof}

\begin{lemma}\label{lemzeta}
For $u\ge 1$, the equation $e^\zeta=-u\zeta$ has a unique solution $\zeta=\mu-\eta i$ in the range $\mu \ge 1$ and $ 0< \eta <3\pi/2$.
This solution satisfies
$$
\log(u\log u)< \mu < \log(u\log u) + \frac{2+\log_2 u}{\log u} \qquad (u\ge 2),
$$
$$
\log(u\log u)< \mu < \log(u\log u) +0.6 \qquad (u\ge 50),
$$
and
$$ \pi <\eta < \pi +  \frac{3 \pi}{2\log u}  \qquad (u > 1).$$
\end{lemma}
\begin{proof}
This follows from the same reasoning as in the proof of \cite[Lem.~4]{Hil},
where a more precise estimate is established for a slightly different $\zeta$. 
\end{proof}

\begin{lemma}\label{lemJ}
For $|\mathrm{Im}(s)| \ge 3$, we have
\begin{equation}\label{eqlemJ}
\left|J(s)-\frac{e^{-s}}{s}\left(1-\frac{1}{s}+\frac{2}{s^2}\right) \right| < 10 \left|\frac{e^{-s}}{s^4}\right|.
\end{equation}
For $\zeta$ as in Lemma \ref{lemzeta}, we have 
$$
\left|J(-\zeta)-u(1+\zeta^{-1}+2\zeta^{-2})\right| < 10 u |\zeta|^{-3} \qquad (u\ge 1)
$$
and
$$
\mathrm{Re}(J(-\zeta)) \ge u \qquad (u\ge  2).
$$
\end{lemma}

\begin{proof}
The second claim follows from the first with $s=-\zeta$. 
The third claim follows from the second for $u\ge 80$, and by graphing for $u\le 80$. 
Repeated integration by parts shows that 
$$
\left|J(s)-\frac{e^{-s}}{s}\left(1-\frac{1}{s}+\frac{2}{s^2}\right) \right| 
\le 6 |e^{-s}| \int_0^\infty \frac{e^{-t}}{|s+t|^4}dt
=:  6 |e^{-s}|   K(s),
$$
say. If $\mathrm{Re}(s)\ge 0$, then $|s+t|\ge |s|$ and $K(s)\le |s|^{-4}$. 
Assume $\mathrm{Re}(s)\le 0$ and write $s=-x+iy$ with $x\ge 0$ and $|y|\ge 3$. 
For $0<a<1$ we have
$$
K(s) \le \int_0^{ax}\frac{e^{-t}}{((1-a)^2 x^2+y^2)^2} dt + \int_{ax}^\infty \frac{e^{-t}}{y^4}dt<\frac{1}{(1-a)^4|s|^4}+\frac{e^{-ax}}{y^4},
$$
so
$$
K(s) < \frac{1}{|s|^4} \left(\frac{1}{(1-a)^4}+ \left(1+\frac{x^2}{3^2}\right)^2 e^{-ax}\right) =: \frac{1}{|s|^4} F(x),
$$
say. Letting $a=0.1$, we find that $F(x)<1.6$ for $x\ge 200$, which is acceptable. 
If $0\le x \le 200$ and $|y|\ge 400$, then $x^2 \le y^2/4$ and 
$$
|s|^4 =(x^2+y^2)^2 \le y^4 (5/4)^2 \le |s+t|^4 (5/4)^2< 1.6 |s+t|^4,
$$
which implies $K(s)\le 1.6|s|^{-4}$. 
In the rectangle $0\le x \le 200$, $3\le y \le 400$, we use a contour plot 
of the quotient of the two sides of \eqref{eqlemJ} to verify the result. 
\end{proof}

\begin{lemma}\label{lemJk}
For $k \in \mathbb{N}$ and $s\in \mathbb{C}\setminus \{0\}$, we have 
$$
J^{(k)}(s) = (-1)^{k} \frac{e^{-s}}{s} \sum_{j=0}^{k-1} \frac{1}{s^j} \frac{(k-1)!}{(k-1-j)!}.
$$
For $k \in \mathbb{N}$, $u\ge 1$ and $\zeta$ as in Lemma \ref{lemzeta}, we have
\begin{equation}\label{eqJk}
J^{(k)}(-\zeta) = (-1)^{k} u \sum_{j=0}^{k-1} \frac{(-1)^j}{\zeta^j} \frac{(k-1)!}{(k-1-j)!}.
\end{equation}
With $\mu=\mathrm{Re}(\zeta)$, we have
$$
|J^{(2)}(-\zeta)| \le u \qquad (u\ge 1),
$$
$$
\mathrm{Re}(J^{(2)}(-\zeta)) \ge u (1-\mu^{-1}) \qquad (u\ge 1),
$$
$$
-\mathrm{Re}(J^{(3)}(-\zeta)) \ge u (1-2\mu^{-1}) \qquad (u\ge 1),
$$
and, for $3\le k \le 8$, $\tau \in \mathbb{R}$ and $u\ge 58$, 
\begin{equation}\label{eqLub}
|J^{(k)}(-\zeta+i\tau)| \le  u \left(1-\frac{k-1}{\mu}+\frac{(k-1)(k-2)+7}{\mu^2}\right).
\end{equation}
\end{lemma}
\begin{proof}
The first equation follows by induction on $k$ from $J'(s)=-e^{-s}/s$. 
Substituting $s=-\zeta$ yields \eqref{eqJk}.
The next three estimates are an easy consequence of \eqref{eqJk}.
To derive \eqref{eqLub}, we differentiate the identity \cite[Eq.~(III.5.44)]{Ten} $k$ times
and apply the triangle inequality to get
$$
|J^{(k)}(-\zeta+i\tau)| \le  \frac{(k-1)!}{|\zeta-i\tau|^k} + |I^{(k)}(\zeta-i\tau)| \qquad (k\ge 1),
$$
where
$$
|I^{(k)}(\zeta-i\tau)| := |\int_0^1 h^{k-1} e^{h(\zeta-i\tau)} dh |\le \int_0^1 h^{k-1} e^{h \mu} dh
=I^{(k)}(\mu).
$$
The upper bound for $I^{(k)}(\mu)$ in \cite[Lemma 1]{Wein} yields
$$
|J^{(k)}(-\zeta+i\tau)| \le  \frac{(k-1)!}{\mu^k} +\frac{e^\mu}{\mu}\left(1-\frac{k-1}{\mu}+\frac{(k-1)(k-2)}{\mu^2}\right) \qquad (k\ge 3).
$$
Writing $\zeta=\mu-\eta i$, we have
$$ \frac{e^\mu}{\mu} =\frac{|e^\zeta|}{\mu}=\frac{ u|\zeta|}{\mu}=u \sqrt{1+\eta^2/\mu^2}\le u \left(1+\frac{\eta^2}{2\mu^2}\right).$$
Inserting this into the previous display and multiplying, we obtain
$$
|J^{(k)}(-\zeta+i\tau)| \le  \frac{(k-1)!}{\mu^k} +u\left(1-\frac{k-1}{\mu}+\frac{(k-1)(k-2)+\eta^2/2}{\mu^2}\right) \quad (k\ge 3),
$$
provided $\mu \ge k-2$, which is implied by $\mu\ge 6$ (i.e. $u\ge 58$) and $k\le 8$. 
Since $\pi <\eta<3.7$ for $u\ge 58$, the result now follows from
$$
 \frac{(k-1)!}{u\mu^{k-2}}+\frac{\eta^2}{2}
< \frac{(k-1)!}{58 \cdot 6^{k-2}}+\frac{3.7^2}{2}
<6.9
< 7 \qquad (3\le k \le 8,\ u\ge 58).
$$
\end{proof}

\begin{lemma}\label{lemReJk}
For $\zeta=\mu-i\eta$ as in Lemma \ref{lemzeta} and $\tau\ge -\eta$ we have
\begin{align*}
\mathrm{Re}(J(-\zeta+i\tau)-J(-\zeta)) & \le -0.4 u \tau^2  & (|\tau|\le 1,\ u\ge 10^3), \\
\mathrm{Re}(J(-\zeta+i\tau)-J(-\zeta)) & \le -0.39 u   &  (1 \le |\tau|\le 4,\ u\ge 50), \\
\mathrm{Re}(J(-\zeta+i\tau)-J(-\zeta)) & \le - \frac{8u}{\log^2 u} &   (\tau \ge 4,\  u \ge 50).
\end{align*}
\end{lemma}
\begin{proof}
For $\tau \ge -\eta$, we have
\begin{equation}\label{eqHt}
H(u,\tau):=J(-\zeta+i\tau)-J(-\zeta)=\int_{-\zeta+i\tau}^{-\zeta} \frac{e^{-s}}{s} ds = i \frac{e^\zeta}{-\zeta} \int_\tau^0 \frac{e^{-it}}{1-it/\zeta} dt.
\end{equation}
Since $e^\zeta=-u\zeta$,
$$
H(u,\tau)= iu \int_\tau^0 e^{-it} dt + iu  \int_\tau^0 e^{-it} \frac{it/\zeta}{1-it/\zeta} dt =: I_1 + I_2,
$$
say. Note that
$
\mathrm{Re}(I_1) = u(\cos \tau -1).
$
To estimate $I_2$, we write $\zeta=\mu-\eta i$ and observe that 
$$
|1-it/\zeta| \ge \mathrm{Re}(1-it/\zeta) = 1+ \frac{t\eta}{\mu^2+\eta^2}\ge  1- \frac{|\tau|\eta}{\mu^2+\eta^2}
> 1- \frac{4|\tau|}{\mu^2+4^2},
$$
since $\eta<4$ by Lemma \ref{lemzeta}. 
If $u\ge 10^3$, then  $|\zeta|\ge \mu >9$ and
$$
\mathrm{Re}(I_2)\le |I_2| \le u \frac{\tau^2/2}{|\zeta|(1-4|\tau|/(\mu^2+4^2))}<u \frac{\tau^2/2}{9(1-4|\tau|/(9^2+4^2))} =: u  f(\tau),
$$
say.

When $|\tau|\le 1$, then $f(\tau)\le 0.058 \tau^2$, while $\cos \tau -1 < -0.459\tau^2$,
which implies that $\mathrm{Re}(I_1+I_2)< -0.4 u \tau^2$, as claimed.

When $1\le |\tau| \le 4$ and $u\ge 10^3$, 
we find that $\cos \tau -1 + f(\tau) < -0.4$, which shows that $\mathrm{Re}(I_1+I_2)< -0.4 u$.
In the finite rectangles $1\le |\tau| \le 4$, $50\le u \le 10^3$, we verify the inequality
$\mathrm{Re}(H(u,\tau))/u \le -0.39$ by graphing the left-hand side. 

To prove the last inequality, we assume that $\tau \ge 4$. Integration by parts applied twice to \eqref{eqHt} shows that 
$$
\frac{H(u,\tau)}{u}=\frac{e^{-i\tau}}{1-i\tau/\zeta} -1 +\frac{1}{\zeta}\left(\frac{e^{-i\tau}}{(1-i\tau/\zeta)^2} -1 \right)
-\frac{2i}{\zeta^2}\int_0^\tau \frac{e^{-it}}{(1-it/\zeta)^3}dt.
$$
Since $|1-it/\zeta|\ge \max(1,t/|\zeta|)$, the modulus of the last integral is $\le \tau$ if $\tau \le |\zeta|$, and if $\tau > |\zeta|$ it is
$$
\le \int_0^\tau \left|\frac{e^{-it}}{(1-it/\zeta)^3}\right|dt \le \int_0^{|\zeta|} 1 dt + \int_{|\zeta|}^\tau \frac{|\zeta|^3}{t^3} dt< \frac{3}{2}|\zeta|.
$$
Thus, either way it is $< \frac{3}{2}|\zeta|$.
If $\tau\ge A|\zeta|$ and $|\zeta|\ge B$ for some constants $A$ and $B$, we obtain
$$
\mathrm{Re}\left( \frac{H(u,\tau)}{u}\right) \le  \frac{1}{A} -1 +\frac{1}{BA^2}+\frac{3}{B},
$$
since $\mathrm{Re}(-1/\zeta) <0$.
If $50\le u \le 10^4$, then $B:=6.9<|\zeta| < 12.3$. If $\tau \ge 370$, then $\tau \ge 30 |\zeta|$. 
With $A:=30$, this shows that 
$$
\mathrm{Re}(H(u,\tau)) < -0.53 u <-\frac{8u}{\log^2 u} \qquad (50\le u \le 10^4, \ \tau \ge 370).
$$
In the rectangle $50\le u \le 10^4$ , $4\le \tau \le 370$, we verify that this inequality holds by graphing 
$\mathrm{Re}(H(u,\tau)) (\log^2 u)/u$, which has a maximum of $-8.08...$ at $(u,t)=(50,6.78...)$. 

It remains to consider $u\ge 10^4$, $\tau \ge 4$. 
We first derive a lower bound for $\mathrm{Re}(J(-\zeta))$. 
Writing $\zeta=\mu-\eta i$, Lemma \ref{lemJ} yields
$$
\mathrm{Re}(J(-\zeta)) \ge u\left(1+ \frac{\mu}{\mu^2+\eta^2}+2\frac{\mu^2-\eta^2}{(\mu^2+\eta^2)^2}-\frac{10}{(\mu^2+\eta^2)^{3/2}}\right).
$$
As $\frac{1}{1+x}>1-x$ for $x>0$, this implies
$$
\mathrm{Re}(J(-\zeta)) \ge u\left(1+ \frac{1}{\mu}+\frac{2}{\mu^2}-\frac{10+\eta^2}{\mu^3}-\frac{6\eta^2}{\mu^4}\right).
$$
Since $\pi < \eta < 3.43$ for $u\ge 10^4$, we get
$$
\mathrm{Re}(J(-\zeta)) \ge u\left(1+ \frac{1}{\mu}+\frac{2}{\mu^2}-\frac{22}{\mu^3}-\frac{71}{\mu^4}\right).
$$

Next, we need an upper bound for $\mathrm{Re}(J(-\zeta+i\tau))$ when $\tau\ge 4$. 
From Lemma \ref{lemJ} we have
$$
\mathrm{Re}(J(-\zeta+i\tau))\le |J(-\zeta+i\tau)| \le u \frac{|\zeta|}{|\zeta-i\tau|}\left(1+\frac{1}{\mu}+\frac{2}{\mu^2}+\frac{10}{\mu^3}\right).
$$
Since $\tau\ge 4$ and $\zeta=\mu-\eta i$, 
$$
\frac{|\zeta|}{|\zeta-i\tau|} 
\le \frac{|\zeta|}{|\zeta- 4i|}
= \sqrt{1-\frac{8\eta+16}{\mu^2+(\eta+4)^2}}
<1-\frac{4\eta+8}{\mu^2+(\eta+4)^2}.
$$
The last expression
is decreasing in $\eta$, since $\pi<\eta\le 3.43$ and $\mu\ge 11.7$ when $u\ge 10^4$. Thus,
$$
\frac{|\zeta|}{|\zeta-i\tau|}< 1-\frac{4\pi+8}{\mu^2+(\pi+4)^2}<1-\frac{14.9}{\mu^2} \qquad (\mu >11.7)
$$
and
$$
\mathrm{Re}(J(-\zeta+i\tau))< u \left(1+\frac{1}{\mu}-\frac{12.9}{\mu^2}-\frac{4.9}{\mu^3}-\frac{29.8}{\mu^4}\right).
$$
Combining this with the lower bound for $\mathrm{Re}(J(-\zeta))$ we get
$$
\mathrm{Re}(J(-\zeta+i\tau))-\mathrm{Re}(J(-\zeta)< u\left(-\frac{14.9}{\mu^2}+\frac{17.1}{\mu^3}+\frac{41.2}{\mu^4}\right)<-\frac{13.1u}{\mu^2},
$$
since $\mu>11.7$. The last expression is $<-8.1u/\log^2 u$ when $u\ge 10^4$, which completes the proof. 
\end{proof}

\begin{lemma}\label{lemtail}
For real $\delta >0$, $A>0$, and $1\le m < \delta^2 A+1$, we have
$$
\int_\delta^\infty t^m e^{-A t^2/2} dt \le \frac{\delta^{m+1}}{\delta^2 A - (m-1)} e^{-A\delta^2/2}.
$$
\end{lemma}
\begin{proof}
This follows from a change of variables in the bound \cite[Prop.~2.7]{Pin} 
$$
\Gamma(a,x):=\int_x^\infty t^{a-1} e^{-t} dt \le \frac{x^a e^{-x}}{x-(a-1)}
$$
for real $x$, $a$ with $x>a-1\ge 0$.
\end{proof}

\subsection{Proof of Theorem \ref{thm2}}
We start with the inverse Laplace integral
\begin{equation*}\label{eqlapinv}
\omega(u)  =\frac{1}{2\pi i} \int_{1-i\infty}^{1+i\infty} \widehat{\omega}(s) e^{us} ds.
\end{equation*}
Writing $s=\sigma+it$, we have $|J(s)| \le e^{-\sigma} |t|^{-1}$, 
which implies $\widehat{\omega}(s) =o(1)$ as $|t|\to \infty$ and $|\sigma|\ll 1$,
by Lemma \ref{lemlap}. 
This allows us to move the line of integration to the left of the origin, picking up the residue of $e^{-\gamma}$, to get
\begin{equation}\label{Wint}
W(u)=\frac{1}{2\pi i} \int_{-\zeta-i\infty}^{-\zeta+i\infty}  (e^{J(s)}-1)e^{us} ds = W^+(u)+W^-(u),
\end{equation}
where $W^+(u)$ and $W^-(u)$ denote the contributions from $\mathrm{Im}(s)>0$ and $\mathrm{Im}(s)<0$,
respectively. Then $W^-(u)=\overline{W^+(u)}$ and 
$$W(u)=W^+(u)+\overline{W^+(u)}= 2 \mathrm{Re}(W^+(u)) .$$
Writing $\mu = \mathrm{Re}(\zeta)$, we have
\begin{equation}\label{eqWp}
W^+(u)=\frac{1}{2\pi } \int_{0}^{\infty}  (e^{J(-\mu+it)}-1)e^{u(-\mu+it)} dt.
\end{equation}

We first consider the contribution to \eqref{eqWp} from $t \ge e^\mu$, where
$$
|J(-\mu+i t)| = \left|\int_0^\infty \frac{e^{\mu - i t -v}}{-\mu+it+v} dv \right|\le \frac{e^{\mu}}{t}\le 1.
$$ 
Approximating the factor $(e^J-1)$ in \eqref{eqWp} by $J$, we have
$$
|e^{J(-\mu+it)} -1 - J(-\mu+it)| \le |J(-\mu+it)|^2 (e-2) \le \frac{e^{2\mu} (e-2)}{t^2}.
$$
The contribution of this last term to \eqref{eqWp} from $t \ge e^\mu$ is
$$
\le \frac{e^{-u\mu+ 2\mu} (e-2)}{2\pi  e^\mu}=
 |\Phi(u)|\frac{e^{\mu}(e-2) \sqrt{2 \pi |J''(-\zeta)|}}{2\pi \exp\{\mathrm{Re}(J(-\zeta))\}}.
$$
Since $\mu \le \log(u\log u)+0.6$ for $u\ge 50$, by Lemma \ref{lemzeta}, $|J''(-\zeta)|\le u$ by Lemma \ref{lemJk} and $\mathrm{Re}(J(-\zeta))\ge u$
by Lemma \ref{lemJ}, this is 
$$
\le  |\Phi(u)|\frac{ e^{0.6} (e-2) u^{3/2}\log u }{\sqrt{2\pi}  e^u} =: |\Phi(u)| E_1(u) \qquad (u\ge 50).
$$
Integration by parts shows that the modulus of the contribution of $J(-\mu+it)$ to \eqref{eqWp} from $t \ge e^\mu$ is
$$
< \frac{e^{-u\mu}}{2 \pi}\cdot \frac{2}{u}
\le  |\Phi(u)|\frac{\sqrt{2 \pi u}}{\pi u e^u}
=: |\Phi(u)| E_2(u)\qquad (u\ge 3).
$$

The modulus of the contribution of the term $-1$ to \eqref{eqWp} from $0 \le t \le e^\mu$ is
$$
\left| - \frac{e^{-u \mu}}{2\pi iu}(e^{iu e^\mu}-1) \right| \le  \frac{e^{-u \mu}}{\pi u} 
\le  |\Phi(u)|\frac{\sqrt{2 \pi u}}{\pi u e^u}
= |\Phi(u)| E_2(u) \qquad (u\ge 2).
$$

It remains to estimate the contribution of $e^{J(-\mu+it)}$ to   \eqref{eqWp} from $0 \le t \le e^\mu$.
Here we change variables via $-\mu+it= -\zeta+i\tau$, i.e. $t = \mathrm{Im}(-\zeta) +\tau = \eta+\tau $. 

Let $\delta:= \sqrt{8\log(u)/u}$. When $u\ge 10^3$, Lemma \ref{lemReJk} shows that
 the contribution of $e^{J(-\mu+i\tau)}$ to   \eqref{eqWp} from $\delta \le |\tau|\le 1$ is
$$
\le \frac{e^{\mathrm{Re} (J(-\zeta)-u\zeta)}}{2\pi} 2 \int_\delta^{1}e^{-0.4 u\tau^2}d\tau
\le \frac{e^{\mathrm{Re} (J(-\zeta)-u\zeta) -0.4u \delta^2}}{2\pi u 0.4\delta },
$$
hence
$$
\le |\Phi(u)| \frac{e^{ -0.4 u \delta^2}}{\sqrt{2\pi u}0.4\delta }=:|\Phi(u)| E_3(u) \qquad (u\ge 10^3).
$$

The contribution  of $e^{J(-\mu+i\tau)}$ to   \eqref{eqWp} from $1 \le |\tau| \le 4$ is 
$$
\le \frac{e^{\mathrm{Re} (J(-\zeta)-u\zeta)}}{2\pi} 2 \int_1^{4}e^{-0.39 u}d\tau
= \frac{3e^{\mathrm{Re} (J(-\zeta)-u\zeta) -0.39u }}{\pi }<e^{\mathrm{Re} (J(-\zeta)-u\zeta) -0.39u },
$$
by Lemma \ref{lemReJk}, hence
$$
<  |\Phi(u)|  \sqrt{2\pi u}e^{-0.39u}=:  |\Phi(u)| E_4(u) \qquad (u\ge 50).
$$

The contribution  of $e^{J(-\mu+i\tau)}$ to   \eqref{eqWp} from $4 \le \tau \le e^\mu -\eta$ is
$$
\le \frac{e^{\mathrm{Re} (J(-\zeta)-u\zeta)}}{2\pi}  \int_4^{e^\mu -3}e^{- \frac{8u}{\log^2 u}}d\tau
= \frac{ (e^ \mu - 7) e^{\mathrm{Re} (J(-\zeta)-u\zeta) - \frac{8u}{\log^2 u} }}{2\pi },
$$
by Lemma \ref{lemReJk}, hence
$$
<  |\Phi(u)|  \frac{\sqrt{ u}  (e^{0.6} u\log(u)- 7) e^{ - \frac{8u}{\log^2 u} }}{\sqrt{2\pi} }=:  |\Phi(u)| E_5(u) \qquad (u\ge 50).
$$

When $-\delta \le \tau \le \delta$, we use the Taylor expansion of $J(-\zeta+i\tau)$ around $\tau=0$.
We write
$$
J_k:=J^{(k)}(-\zeta) \quad (2 \le k \le 7), \qquad J_8:=\max_{|\tau|\le \delta}|J^{(8)}(-\zeta +i\tau)|.
$$
 Since
 $J'(-\zeta)=-u$,
$$
J(-\zeta+i\tau)-J(-\zeta)+ui\tau = \sum_{k=2}^7 \frac{J_k}{k!}(i\tau)^k + \frac{J_8\ \kappa_{u,\tau}}{8!} \tau^8,
$$
where $|\kappa_{\tau,u}|\le 1$. 
Thus,
\begin{equation}\label{eqhint}
\int_{-\delta}^{\delta} e^{J(-\zeta+i\tau)-J(-\zeta)+ui\tau }d\tau = \int_{-\delta}^\delta e^{-J_2\tau^2/2}h(u,\tau)d\tau, 
\end{equation}
where 
$$
h(u,\tau):= \exp\left\{ \sum_{k=3}^7 \frac{J_k}{k!}(i\tau)^k + \frac{J_8\kappa_{u,\tau}}{8!} \tau^8\right\}.
$$
Writing $M$ for the exponent, we have
$$
h(u,\tau)=e^M=\sum_{j=0}^5 \frac{M^j}{j!}+\frac{e^{1/4} |M|^6 \nu_M}{6!} \qquad (\mathrm{Re}( M)  \le 1/4),
$$
where $|\nu_M|\le 1$. 
We claim that 
$$
\mathrm{Re}( M)  \le |\mathrm{Im}( J_3)|  \frac{\delta^3}{3!}+\sum_{k=4}^8 \frac{|J_k|}{k!}\delta^k < \frac{1}{4}
\qquad (u\ge 860).
$$
For $u\ge 10^6$, this follows from $ |\mathrm{Im}( J_3)| < |J_3|$ and  \eqref{eqLub};
when $860\le u \le 10^6$, we graph the expression of the 
last display, with $|J_8|$ replaced by its upper bound in \eqref{eqLub}, to verify this claim.

Let $h_k$ be $M^k$ without all the terms that contribute zero to the integral in \eqref{eqhint} and write
$$
 a_k:= \frac{J^{(k)}(-\zeta)}{k!}=\frac{J_k}{k!}, \qquad L_k:=|J_k|.
$$
We have
$$h_0=1,$$ 
$$
h_1=a_4 \tau^4 - a_6  \tau^6 +   \theta_1 \tau^8 R_1(\tau),
$$
$$
h_2= -a_3^2 \tau^6 +2a_3a_5\tau^8 +a_4^2 \tau^8
+\theta_2  \tau^{10} R_2(\tau),
$$
$$
h_3=-3a_3^2 a_4\tau^{10}  +  \theta_3   \tau^{12} R_3(\tau)
$$
$$
h_4=a_3^4 \tau^{12} + \theta_4  \tau^{14} R_4(\tau),
$$
$$
h_5 =\theta_5 \tau^{16} R_5(\tau),
$$
where $|\theta_i|\le 1$, $R_1(\tau):=L_8/8!$, 
$$
R_2(\tau):= \sum_{3\le i,j\le 8, \ i+j \ge 10 \atop  i+j \equiv 0 \bmod 2 \text{ or } \max(i,j)=8}\frac{\tau^{i+j-10}}{i!j!}L_i L_j,
$$
$$
R_3(\tau):=\sum_{3\le i,j,k\le 8, \ i+j+k \ge 12 \atop  i+j +k\equiv 0 \bmod 2 \text{ or } \max(i,j,k)=8}\frac{\tau^{i+j+k-12}}{i!j!k!} L_i L_j L_k,
$$
et cetera.

Define 
$$K_2:=\mathrm{Re}(J_2) = \mathrm{Re}(J''(-\zeta))   \ge u(1-1/\mu),$$
by Lemma \ref{lemJk}.
With the help of the upper bound \eqref{eqLub} when $\mu\ge 16$ (where the right-hand side of \eqref{eqLub}
is decreasing in $k$), and by graphing the exact values when $\mu\le 16$, 
we verify that $L_8 < K_2^4 u^{-3}$ for $u\ge 1$,
and each product $L_i L_j$ appearing in $R_2$ satisfies $L_i L_j < K_2^5  u^{-3}$ for $u\ge 1$, et cetera, so
that each product of $L$'s appearing in $R_j$ is $<K_2^{j+3} u^{-3}$ for $1\le j \le 5$ and $u\ge 1$. 
Since $|\tau|\le \delta \le 1/4$, which holds for $u\ge 870$, we have $|R_j(\tau)| \le R_j(1/4) $.
Let $ \gamma_j$ denote the value of the
sum $R_j(1/4)$, but without the products of the $L$'s. Then
$$
 |R_j(\tau)| \le \gamma_j  K_2^{j+3} u^{-3}.
$$

We claim that $L_{k}\le L_3$ for $4\le k \le 8$ and $u\ge 10^3$. This follows from applying \eqref{eqLub} for $4\le k \le 8$. For
$L_3$, the lower bound $L_3\ge |\mathrm{Re}(J_3)|>u(1-2/\mu)$ from Lemma \ref{lemJk} suffices for $u\ge 4\cdot 10^4$, 
while for $10^3\le u \le  4\cdot 10^4$ we graph
the exact value of $L_3=|J_3|$ and the upper bound \eqref{eqLub} for $L_k$.  Thus,
$$
|M|\le \sum_{k=3}^8 \frac{L_k}{k!}|\tau|^k \le \frac{L_3 |\tau|^3}{3!} 1.065761 \quad  (|\tau|\le \delta \le 1/4).
$$
This implies 
$$
|e^{1/4} M^6| \le \frac{L_3^6 \tau^{18}}{(3!)^6} e^{1/4} 1.4655 <  \frac{u^{-3} K_2^9 \tau^{18}}{(3!)^6} e^{1/4} 1.4655
 =: u^{-3} K_2^9 \tau^{18} \gamma_6,
$$
since $L_3^6< u^{-3}K_2^9$ for $u\ge 1$, which is verified as above.
The contribution of the $R_j$'s and $|e^{1/4} M^6|$ to \eqref{eqhint} is 
\begin{multline*}
< \sum_{j=1}^6 \frac{\gamma_j  K_2^{3+j}}{u^3 j!} \int_{-\infty}^\infty e^{-K_2\tau^2/2} \tau^{6+2j} d\tau
= \sqrt{\frac{2\pi}{K_2}}\sum_{j=1}^6 \frac{\gamma_j  K_2^{3+j}}{u^3 j!} \frac{(5+2j)!!}{K_2^{3+j}}\\
= \sqrt{\frac{2\pi}{K_2}}\frac{1}{u^3}\sum_{j=1}^6 \frac{\gamma_j (5+2j)!!}{j!}
< \sqrt{\frac{2\pi}{K_2}} \frac{8.29}{u^3}< \sqrt{\frac{2\pi}{|J_2|}} \frac{8.3}{u^3},
\end{multline*}
where $n!!$ is the product of all the positive integers up to $n$ that have the same parity as $n$.
The last estimate follows from 
$$\frac{|J_2|}{K_2}=\sqrt{1+\frac{\eta^2}{(\mu^2+\eta^2-\mu)^2}} \le 
1+\frac{4^2}{2(9.1^2+4^2-9.1)^2} < 1.001 \qquad (u\ge 10^3). 
$$
Writing the main terms as 
$g(\tau):=1+\sum_{j=2}^6 c_j(u) \tau^{2j}$, Lemma \ref{lemtail} yields
$$
\int_\delta^\infty  e^{-K_2\tau^2/2} |g(\tau)| d\tau 
\le \frac{e^{-K_2 \delta^2/2}}{K_2\delta} \left(1 + \frac{3}{2}\sum_{j=2}^6 |c_j(u)| \delta^{2j} \right) , 
$$
provided $K_2 \delta^2 \ge 33$, which is the case if $u\ge 132$. 
With \eqref{eqLub} we find that
$$
1 + \frac{3}{2}\sum_{j=2}^6 |c_j(u)| \delta^{2j} \le 6
$$
for $u\ge 770$, which implies
$$
2\int_\delta^\infty  e^{-K_2\tau^2/2}|g(\tau)| d\tau 
<\frac{12 e^{-K_2 \delta^2/2}}{\delta K_2} = \frac{ 12e^{-K_2 \delta^2/2} \sqrt{|J_2|}}{ \delta K_2  \sqrt{2\pi}} \sqrt{\frac{2\pi}{ |J_2|}}  .
$$
Since $K_2 \ge u(1-1/\mu)\ge u(1-1/L)$ with $L:=\log(u\log u)$, and $|J_2| \le u$, by Lemmas \ref{lemzeta} and \ref{lemJk}, this is
$$
< \frac{ 12e^{-u(1-1/L)\delta^2/2} \sqrt{u}}{ \delta u(1-1/L)  \sqrt{2\pi}} \sqrt{\frac{2\pi}{ |J_2|}} 
=: E_6(u)  \sqrt{\frac{2\pi}{ |J_2|}} \qquad (u\ge 770).
$$

Now $f(u):=u^3(E_2(u)+\sum_{k=1}^6 E_k(u))=u^{3-(0.4)8+o(1)}=u^{-0.2+o(1)}$ has a limit of zero and is decreasing for $u\ge 10$, while 
$f(730)<0.06$. 

Combining everything, we obtain 
\begin{equation}\label{eqfin}
W^+(u) = \Phi(u) \left(1+\alpha(u) +\beta(u)+\lambda(u) u^{-3}\right),
\end{equation}
where $|\lambda(u)|<8.3 + 0.06 =8.36$ for $u\ge 10^3$,
$$
\alpha(u) := \frac{a_4 3!!}{J_2^2} - \frac{a_3^2 5!!}{2! J_2^3}
=\frac{J_4}{8 J_2^2}-\frac{5J_3^2}{24 J_2^3}  
=-\frac{ \zeta^4-4 \zeta^3+\frac{13}{2} \zeta^2-2 \zeta+1}{ 12u \zeta (\zeta-1)^3 }
\sim -\frac{1}{12u},
$$
\begin{equation*}
\begin{split}
\beta(u):= & -\frac{a_6 5!!}{J_2^3}+\frac{(2 a_3 a_5 + a_4^2) 7!!}{2! J_2^4}- \frac{3a_3^2 a_4 9!!}{3! J_2^5}
+\frac{a_3^4 11!!}{4! J_2^6}\\
= & 
-\frac{J_6}{48 J_2^3}
+\frac{7 J_3 J_5}{48  J_2^4}
+\frac{35 J_4^2}{384 J_2^4}
-\frac{35 J_3^2 J_4}{64 J_2^5}
+\frac{385 J_3^4}{1152 J_2^6}\\
= & \frac{\zeta^8-8 \zeta^7+17 \zeta^6+4\zeta^5+\frac{25}{4} \zeta^4+62 \zeta^3+17\zeta^2-4 \zeta+1}{288u^2 \zeta^2 (\zeta-1)^6}
 \sim \frac{1}{288 u^2}.
\end{split}
\end{equation*}

Note that $u^2 \beta(u)=  \frac{1}{288}(1-2/\zeta +O(1/\zeta^2))$ and
$\lim_{u\to \infty} u^2 \beta(u) = \frac{1}{288}$. 
Moreover, $u^2 |\beta(u)|$ is increasing for $u\ge 85$,
so $|u^2 \beta(u) |< \frac{1}{288}$ for $u\ge 85$. 
From \eqref{eqfin} we obtain, for $u\ge 10^3$,
$$
W^+(u)=\Phi_2(u)\{1+ (\beta(u)+\lambda(u)u^{-3})/(1+\alpha(u))\},
$$
which implies that
\begin{equation}\label{eqtheta2}
|\theta_2(u)|\le |\beta(u) + \lambda(u) u^{-3} | /|1+\alpha(u)| <\left(\frac{1}{288u^2} + \frac{8.36}{u^3}\right)/\left(1-\frac{1}{10u}\right),
\end{equation}
since $|1+\alpha(u)|>1-\frac{1}{10u}$ by Lemma \ref{lemalpha}.
Thus, $|\theta_2(u)|<0.005/u^2$ if $u\ge 5500$. 
For $6\le u \le 5500$, we calculate $W^+(u)$ numerically as described in Section \ref{SecNum}, 
to find that Theorem \ref{thm2} also holds in this range.
This completes the proof of Theorem \ref{thm2}.

In proving Theorem \ref{thm2}, we also proved Theorem \ref{thm3}, which follows from \eqref{eqfin} for $u\ge 10^3$.
The range $1\le u\le 10^3$ is again established in Section \ref{SecNum}. 

\begin{theorem}\label{thm3}
Let $\Phi_3(u):=\Phi(u)(1+\alpha(u)+\beta(u))$. 
We have 
$$
W(u)=2 |\Phi_3(u)|\left\{\cos(\arg(\Phi_3(u)))+\theta_3(u)\right\},
$$
where $|\theta_3(u)|< 0.01 u^{-3}$ for $18 \le u \le 10^3$ and $|\theta_3(u)|<8.4 u^{-3}$ for $u\ge 10^3$.
\end{theorem}

It appears that the factor $8.4$ is far from best possible and that $|\theta_3(u)|< 0.01 u^{-3}$ likely holds for all $u\ge 18$.

\section{Proof of Theorem \ref{thm1} and Corollary \ref{cor1}}

\begin{lemma}\label{lemalpha}
For $u\ge 2$ we have 
$$
\left| \alpha(u) +\frac{1}{12u}\right|< \frac{1}{12 u \log(u\log u)}.
$$
\begin{proof}
From Lemma \ref{lemJk} we get
$$
12 \zeta u\left(\alpha(u)+\frac{1}{12u}\right)=1-\frac{\zeta (\zeta+4)}{2 (\zeta-1)^3}=:f(\zeta),
$$
say. An exercise shows that $|f(x+iy)|<1$ for $x\ge 5$ and $-4 \le y \le 4$. 
Indeed, when $x\ge 10$ this follows from $|1-f(x+iy)|<1/x$ and $|\arg(1-f(x+iy))| < \pi/5$,
while for $5\le x \le 10$ we verify that $|f(x+iy)|<1$ with a contour plot.
 
From Lemma \ref{lemzeta} we find that
 $\mathrm{Re}(\zeta)>5$,
$-4<\mathrm{Im}(\zeta)<-3$, and 
$|\zeta|>\log(u\log u)$,  for $u\ge 25$. This establishes the claim when $u\ge 25$. 
For $2 \le u \le 25$, we verify the inequality by graphing both sides.
\end{proof}
\end{lemma}

\begin{proof}[Proof of Theorem \ref{thm1}]
Equation \eqref{eqfin} shows that $|\theta_1(u)|\le |\nu(u)|$, where
$$
\nu(u):=\frac{\frac{1}{12u}+\alpha(u) + \beta(u) +\lambda(u)u^{-3}}{1-\frac{1}{12u}}.
$$
From Lemma \ref{lemalpha} and \eqref{eqtheta2} we get, for $u\ge 10^3$,
$$
|\nu(u)| \le \frac{1}{12 u \log(u\log u)(1-\frac{1}{12u})}+\frac{0.012}{u^2}< \frac{1}{12 u \log u}.
$$
For $u\le 10^3$, we verify Theorem \ref{thm1} as described in Section \ref{SecNum}. 
\end{proof}

\begin{proof}[Proof of  Corollary \ref{cor1}]
When $u\ge 6$, Corollary \ref{cor1} follows from Theorem \ref{thm1}. 
For $3\le u \le 6$, we calculate $W(u)$ as described in Section \ref{SecNum} and graph $|W(u)/2\Phi(u)|$.
\end{proof}

\section{Numerical calculation of $W(u)$}\label{SecNum}

When $ u \le 50$, we verify the three theorems by calculating $W(u)$ with the algorithm of Marsaglia, Zaman and Marsaglia \cite{MZM},
and graphing $\theta_i(u)$ for $i=1,2,3$.   
To approximate $W(u)$ in the interval $[n,n+1]$, where $n$ is an integer, 
we find the coefficients of the Taylor polynomial of order $250$ centered
at $n+1/2$, recursively from the coefficients of the polynomial centered at $n-1/2$.

For $u\ge 50$, we evaluate the integral in \eqref{eqWp} numerically. 
As in the proof of Theorem \ref{thm2}, we first truncate the integral and estimate the contribution of $-1$, to get
$$
|W^+(u)-\tilde{W}^+(u)|< |\Phi(u)| E(u),
$$
where
$$
\tilde{W}^+(u):=\frac{1}{2\pi } \int_{-1}^{1}  e^{J(-\zeta+i\tau)+u(-\zeta+i\tau)} d\tau
$$
and
$$
E(u):=E_1(u)+2E_2(u)+E_4(u)+E_5(u).
$$
We rewrite the integral as
$$
\tilde{W}^+(u)=\frac{e^{J(-\zeta)-u\zeta}}{\sqrt{2\pi u}}
 \int_{-\sqrt{u}}^{\sqrt{u}} \frac{e^{J(-\zeta+i\tau/\sqrt{u})-J(-\zeta)+iu\tau/\sqrt{u}}}{\sqrt{2\pi}} d\tau
=:\frac{e^{J(-\zeta)-u\zeta}}{\sqrt{2\pi u}} S(u),
$$
say.
The proof of Theorem \ref{thm2} shows that the last integrand approaches the standard normal density function as $u$ grows.
Since 
$$
|\theta_1(u)|=\frac{|\mathrm{Re}(W^+(u)-\Phi_1(u))|}{|\Phi_1(u)|}
\le \left| \frac{W^+(u)}{\Phi_1(u)} -1 \right|
\le \left| \frac{\tilde{W}^+(u)}{\Phi_1(u)} -1 \right| +\left| \frac{\Phi(u) E(u)}{\Phi_1(u)}\right|,
$$
we get
$$
|\theta_1(u)|\le \left|\frac{S(u) \sqrt{1-1/\zeta}}{1-\frac{1}{12u}} -1\right|
+\left|\frac{E(u)}{1-\frac{1}{12u}}\right|=:H_1(u),
$$
say.
Using numerical integration in Mathematica to evaluate $S(u)$, we 
verify Theorem \ref{thm1} for $50\le u \le 10^3$ by graphing $H_1(u)$ and $\frac{1}{12u\log u}$. 
Replacing $1-\frac{1}{12u}$ by $1+\alpha(u)$ in the last display,
we verify Theorem \ref{thm2} for $50\le u \le 5500$. 
Theorem \ref{thm3} is confirmed similarly for $50\le u \le 10^3$, with $1+\alpha(u)+\beta(u)$ replacing $1-\frac{1}{12u}$.

\end{document}